\documentclass{amsart}

\newcommand{\w}{\omega}
\newcommand{\U}{\mathcal U}
\newcommand{\V}{\mathcal V}
\newcommand{\W}{\mathcal W}
\newcommand{\IN}{\mathbb N}
\newcommand{\IZ}{\mathbb Z}
\newcommand{\St}{\mathrm{St}}
\newcommand{\dowa}{{\downarrow}}

\newtheorem{theorem}{Theorem}[section]
\newtheorem{corollary}[theorem]{Corollary}
\newtheorem{claim}[theorem]{Claim}
\newtheorem{example}[theorem]{Example}

\title{Metrizability of Clifford topological semigroups}
\author{Taras Banakh, Oleg Gutik, Oles Potiatynyk, Alex Ravsky}
\address{Ivan Franko National University of Lviv, Ukraine}
\subjclass{22A15; 54E35; 54E18; 54D30}
\email{t.o.banakh@gmail.com, ovgutik@yahoo.com, oles2008@gmail.com, oravsky@mail.ru}

\begin{document}
\begin{abstract}
We prove that a topological Clifford semigroup $S$ is metrizable if and only if $S$ is an $M$-space and the set $E=\{e\in S:ee=e\}$ of idempotents of $S$ is a metrizable $G_\delta$-set in $S$. The same metrization criterion  holds also for any countably compact Clifford topological semigroup $S$.
\end{abstract}

\maketitle

\section*{Introduction}

According to the classical Birkhoff-Kakutani Theorem \cite[3.3.12]{AT}, a topological group $G$ is metrizable if and only if it is first countable. In this paper we establish metrization criteria for topological Clifford semigroups. In particular, in Theorem~\ref{t2.4} we prove that a topological Clifford semigroup is metrizable if and only if $S$ is an $M$-space and the subset of idempotents $E=\{e\in S:ee=e\}$ is a metrizable $G_\delta$-set in $S$. In Theorem~\ref{t3.1} we shall prove that a countably compact Clifford topological semigroup $S$ is metrizable if and only if $E$ is a metrizable $G_\delta$-set in $S$. All topological spaces considered in this paper are  regular.

\section{Clifford topological semigroups versus topological Clifford semigroups}

In this section we discuss the interplay between Clifford topological semigroups and topological Clifford semigroups. A {\em topological semigroup} is a topological space $S$ endowed with a continuous associative operation $\cdot:S\times S\to S$.

A topological semigroup $S$ is a {\em Clifford topological semigroup} if it is algebraically Clifford, i.e., $S$ is the union of groups. For each element $x\in S$ of a Clifford semigroup there is a unique element $x^{-1}\in S$ such that $xx^{-1}x=x$, $x^{-1}xx^{-1}=x^{-1}$ and $xx^{-1}=x^{-1}x$. This element $x^{-1}$ is called the {\em inverse} of $x$. The map $(\cdot)^{-1}:S\to S$, $(\cdot)^{-1}:x\mapsto x^{-1}$, is called the {\em inversion} on $S$. In Clifford topological semigroups the inversion is not necessarily continuous.

By a {\em topological Clifford semigroup} we mean a Clifford topological semigroup $S$ with continuous inversion $(\cdot)^{-1}:S\to S$. The simplest example of a Clifford topological semigroup which is not a topological Clifford semigroup is the half-line $[0,\infty)$ endowed with the operation of multiplication of real numbers. This semigroup is locally compact. Such examples cannot occur among compact topological semigroups because of the following classical result whose proof can be found in \cite{KW} or \cite{Kru}.

\begin{theorem}\label{t1.1} For each compact Clifford topological semigroup $S$ the inversion operation $(\cdot)^{-1}:S\to S$ is continuous, which implies that $S$ is a topological Clifford semigroup.
\end{theorem}

For countably compact Clifford topological semigroup the inversion operation is sequentially continuous.
Let us recall that a function $f:X\to Y$ between topological spaces is {\em sequentially continuous} if for each convergent sequence $(x_n)_{n=1}^\infty$ in $X$ the sequence $\big(f(x_n)\big)_{n=1}^\infty$ is convergent in $Y$ and $\lim\limits_{n\to\infty} f(x_n)=\lim\limits_{n\to\infty}f(x_n)$.

%Each sequentially continuous function $f:X\to Y$ defined on a sequential topological space $X$ is continuous. We recall that a topological space $X$ is sequential if each sequentially open subset of $X$ is open. A subset $U\subset X$ is {\em sequentially open} if for each sequence $\{x_n\}_{n\in\w}\subset X$ that converges to a point $x_\infty\in U$ there is a number $n_0\in\w$ such that $x_n\in U$ for all $n\ge n_0$.

The following ``countably compact'' version of Theorem~\ref{t1.1} was proved by Gutik, Pagon, and Repov\v s in \cite{GR}.

\begin{theorem}\label{t1.2} For each countably compact Clifford topological semigroup $S$ the inversion operation $(\cdot)^{-1}:S\to S$ is sequentially continuous.
\end{theorem}

There are also some other conditions guaranteeing that a countably compact Clifford topological semigroup is a topological Clifford semigroup. One of such conditions is the topological periodicity.

A topological semigroup $S$ is called {\em topologically periodic} if each element $x\in S$ is {\em topologically periodic} in the sense that for each neighborhood $O_x\subset S$ of $x$ there is an integer $n\ge 2$ with $x^n\in O_x$.

The proof of the following criterion can be found in \cite{GR}:

\begin{theorem}\label{t1.3} A countably compact Clifford topological semigroup $S$ is a topological Clifford semigroup if one of the following conditions is satisfied:
\begin{enumerate}
\item the space $S$ is Tychonoff and the square $S\times S$ is pseudocompact;
\item the square $S\times S$ is countably compact;
\item $S$ is sequential;
\item $S$ is topologically periodic and first countable at each idempotent $e\in E$.
\end{enumerate}
\end{theorem}

For inverse Clifford topological semigroups this theorem was proved in \cite{BG}.

\section{Cardinal characteristics of topological Clifford semigroups}

In this section we evaluate some cardinal characteristics of topological Clifford semigroups.

For a topological space $X$ and a subset $A\subset X$ we shall be interested in the following  cardinal characteristics:
\begin{itemize}
\item the {\em weight} $w(X)$ of $X$, equal to the smallest infinite cardinal $\kappa$ for which there is a base $\mathcal B$ of the topology of $X$ with $|\mathcal B|\le\kappa$;
\item the {\em Lindel\"of number $l(X)$}, equal to the smallest infinite cardinal $\kappa$ such that each open cover $\U$ of $X$ has a subcover $\V$ of cardinality $|\V|\le\kappa$;
\item the {\em pseudocharacter} $\psi(A,X)$ of $A$ in $X$, equal to the smallest cardinality $|\mathcal U|$ of a family $\U$ of open subsets of $X$ such that $\cap\U=A$;
\item the {\em diagonal number} $\Delta(X)=\psi(\Delta_X,X\times X)$, equal to the pseudocharacter of the diagonal $\Delta_X=\{(x,x):x\in X\}$ in the square $X\times X$.
\end{itemize}

We say that a topological space $X$ has {\em $G_\delta$-diagonal} if $\Delta(X)\le\aleph_0$. It is easy to see that $\Delta(X)\le w(X)$. By a result of Arkhangelski \cite[II.\S1]{Ar}, each locally compact space $X$ has weight $w(X)=l(X)\cdot \Delta(X)$. In particular, each Lindel\"of locally compact space with $G_\delta$-diagonal is metrizable. For topological inverse Clifford semigroups the following theorem was proved in \cite[2.3(10)]{Ba}.

\begin{theorem}\label{t2.1} For a topological Clifford semigroup $S$ and its subset of idempotents  $E=\{e\in S:ee=e\}$ we have the upper bound $\Delta(S)\le\Delta(E)\cdot\psi(E,S)$.
\end{theorem}

\begin{proof} By the definition of $\Delta(E)$, there is a family $\U$ of open neighborhoods of the diagonal $\Delta_E$ in $E\times E$ such that $|\U|=\Delta(E)$ and $\bigcap\U=\Delta_E$. By a similar reason, there is a family $\V$ of open neighborhoods of the set $E$ in $S$ such that $|\V|=\psi(E,S)$ and $\cap\V=E$.

It follows from the continuity of the multiplication and the inversion on $S$ that for any open sets $U\in\U$ and $V\in\V$ the set
$$W_{U,V}=\{(x,y)\in S\times S: (xx^{-1},yy^{-1})\in U,\;\;xy^{-1}\in V\}$$is open in $S\times S$. Consider the family $\W=\{W_{U,V}:U\in\U,\;V\in\V\}$ and observe that $\Delta_S=\bigcap\W$ and hence $\Delta(S)\le|\W|\le|\U|\cdot|\V|=\Delta(E)\cdot\psi(E,S)$.
\end{proof}

Combining Theorem~\ref{t2.1} with the equality $w(X)=l(X)\cdot\Delta(X)$ holding for each locally compact space $X$ (see \cite{Ar}), we get the following

\begin{corollary}\label{c2.2} Each locally compact topological Clifford semigroup $S$ has weight $w(S)=l(S)\cdot w(E)\cdot\psi(E,S)$.
\end{corollary}

This corollary implies the following metrizability criterion:

\begin{corollary} A Lindel\"of locally compact topological Clifford semigroup $S$ is metrizable if and only if the set of idempotents $E$ is a metrizable $G_\delta$-set in $S$.
\end{corollary}

In fact, this metrizability criterion holds more generally for topological Clifford semigroups which are $M$-spaces. The class of $M$-spaces includes all metrizable spaces, all Lindel\"of locally compact spaces, and all countably compact spaces, see \cite[\S3.5]{Gru}. Let us recall \cite[3.5]{Gru} that a topological space $X$ is called an {\em $M$-space} if
there is a sequence $(\U_n)_{n\in\w}$ of open covers of $X$ such that each cover $\U_{n+1}$ star refines the cover $\U_{n+1}$ and for any point $x\in X$, any sequence $x_n\in \St(x,\U_n)$, $n\in\w$, has a cluster point in $X$. By a characterization theorem of Morita \cite{Mor} (see also \cite[3.6]{Gru}), a topological space $X$ is an M-space if and only if it admits a closed continuous map $f:X\to Y$ onto a metrizable space $Y$ with countably compact preimages $f^{-1}(y)$ of points $y\in Y$.
By \cite[3.8]{Gru}, each $M$-space with $G_\delta$-diagonal is metrizable. This fact combined with Theorem~\ref{t2.1} implies the following metrization theorem for topological Clifford semigroups.

\begin{theorem}\label{t2.4} A topological Clifford semigroup $S$ is metrizable if and only if $S$ is an $M$-space and $E$ is a metrizable $G_\delta$-set in $S$.
\end{theorem}

For topological inverse Clifford semigroups this metrizability criterion was proved in \cite{Ba}. The following example shows that Theorem~\ref{t2.4} does not hold without the M-space assumption.

\begin{example} There is a non-metrizable countable commutative topological Clifford semigroup $S$ such that its set of idempotents $E$ is compact metrizable and open in $S$.
\end{example}

\begin{proof} Let $X$ be any countable topological space with a unique non-isolated point $x_0$. Define a continuous semilattice operation $\wedge$ on $X$ letting
$$x\wedge y=\begin{cases}x&\mbox{if $x=y$}\\
x_0&\mbox{if $x\ne y$}.\end{cases}
$$
Let $\tilde X$ be the space $X$ endowed with the topology of one-point compactification of the countable discrete space $X'=X\setminus\{x_0\}$.

Consider the commutative Clifford semigroup $S=X\times \IZ$ endowed with the topology of the topological sum $(X\times (\IZ\setminus\{0\})\oplus(\tilde X\times\{0\})$. Here $\IZ$ is the discrete additive group of integers. It is easy to check that $S$ is a topological Clifford semigroup whose set of idempotents $E=\tilde X\times\{0\}$ is metrizable, compact  and open in $S$. The semigroup $S$ is metrizable if and only if so is the space $X$.
\end{proof}

\section{Metrizability of countably compact Clifford topological semigroups}

Theorem~\ref{t2.4} implies that a countably compact topological Clifford semigroup $S$ is metrizable if and only if the set $E$ of idempotents of $S$ is a metrizable $G_\delta$-set in $S$. In this section we shall generalize this metrization criterium to countably compact Clifford topological semigroups.

\begin{theorem}\label{t3.1} A countably compact Clifford topological semigroup $S$ is metrizable if and only if the set $E$ of idempotents of $S$ is a metrizable $G_\delta$-set in $S$.
\end{theorem}

\begin{proof} Assume that $E$ is a metrizable $G_\delta$-set in $S$. The continuity of the semigroup operation on $S$ implies that the set of idempotents $E=\{e\in S:ee=e\}$ is closed in $S$ and hence is countably compact. Being metrizable, the countably compact space $E$ is compact.

\begin{claim}\label{cl1} The space $S$ is first countable at each point $e\in E$.
\end{claim}

\begin{proof} Using the fact that $E$ is a metrizable $G_\delta$-set in $S$, it can be shown that each singleton $\{e\}\subset E$ is a $G_\delta$-set in $S$. Using the regularity of the space $S$, we can choose a countable decreasing sequence $(U_n)_{n\in\w}$ of open subsets of $S$ such that $\bigcap_{n\in\w}\overline{U}_n=\{e\}$.
The countable compactness of $S$ implies that for each neighborhood $U\subset S$ of $e$ there is $n\in\w$ with $\overline{U}_n\subset U$. This means that $\{U_n\}_{n\in\w}$ is a neighborhood base at $e$, so $S$ is first countable at $e$.
\end{proof}

\begin{claim}\label{cl2} The inversion $(\cdot)^{-1}:S\to S$ is continuous at each idempotent $e\in E$.
\end{claim}

\begin{proof} By Theorem~\ref{t1.2}, the inversion operation is sequentially continuous at hence is continuous at each point $x\in S$ having countable neighborhood base. In particular, the inversion is continuous at each idempotent $e\in E$.
\end{proof}

\begin{claim}\label{cl3} For every idempotent $e\in E$ the maximal subgroup $$H_e=\{x\in S:xx^{-1}=e\}$$
is a metrizable topological group.
\end{claim}

\begin{proof} The continuity of the semigroup operation on $S$ implies that $H_e$ is a paratopological group.  By Claim~\ref{cl2}, the inversion of $S$ is continuous at the idempotent $e$. Consequently, the inversion of the paratopological group $H_e$ is continuous at $e$ and so $H_e$ is a topological group, see \cite[1.3.12]{AT}. Since $S$ is first countable at $e$, the paratopological group $H_e$ is first countable. By the Birkhoff-Kakutani Theorem \cite[3.3.12]{AT}, the topological group $H_e$ is metrizable.
\end{proof}

\begin{claim}\label{cl4} The semigroup $S$ is topologically periodic.
\end{claim}

\begin{proof} Take any point $a\in S$ and denote by $A$ the (closed) set of accumulation points of the sequence $\{a^n:n\in\IN\}$ in $S$. The continuity of the semigroup operation implies that $A$ is a closed commutative subsemigroup in $S$.
Let $\pi:S\to E$, $\pi:x\mapsto xx^{-1}=x^{-1}x$, denote the projection of $S$ onto the set $E$ of idempotents of $S$. Observe that the map $\pi$ is not necessarily continuous.

We claim that the projection $\pi(A)\subset E$ has a minimal element with respect to the natural partial order on $E$ defined by $x\le y$ iff $xy=x=yx$. By Zorn Lemma, it suffices to prove that each linearly ordered subset $L$ of $\pi(A)$ has a lower bound in $\pi(A)$. For each element $\lambda\in L$ consider its lower cone $${\dowa}\lambda=\{e\in E:e\le \lambda\}=\{e\in E:\lambda e\lambda=e\}$$ and observe that it is a closed subset of $E$. Next, consider the closed subset
$${\dowa}L=\bigcap_{\lambda\in L}{\dowa}\lambda.$$ The compactness of $E$ implies that each open neighborhood of ${\dowa}L$ contain some lower cone ${\dowa}\lambda$, $\lambda\in L$. Since ${\dowa}L$ is a closed $G_\delta$-set in $E$, there is a decreasing sequence $\{\lambda_n\}_{n\in\w}\subset L$ such that ${\dowa}L=\bigcap_{n\in\w}{\dowa}\lambda_n$.
For every $n\in\w$ choose a point $a_n\in A$ with $\pi(a_n)=\lambda_n$ and observe that $\lambda_na_n\lambda_n=(a_na_n^{-1})a_n(a_n^{-1}a_n)=a_n$ and hence $a_n$ lies in the closed subset
$$\lambda_nS\lambda_n=\{x\in S:\lambda_nx\lambda_n=x\}$$of $S$. Then for every $m\ge n$ we get $\lambda_m\le\lambda_n$ and hence $a_m\in\lambda_mS\lambda_m\subset \lambda_nS\lambda_n$.

By the countable compactness of $A$, the sequence $(a_n)_{n\in\w}$ has an accumulation point $a_\infty$ which lies in the closed subset $\bigcap_{n\in\w}\lambda_n S\lambda_n$ of $S$. Let $\lambda_\infty=\pi(x_\infty)\in\pi(A)$. We claim that $\lambda_\infty\in\bigcap_{n\in\w}{\dowa}\lambda_n={\dowa}L$, which means that $\lambda_\infty$ is a lower bound for $L$ in $\pi(A)$.

Indeed, for every $n\in\w$, the inclusion $a_\infty\in\lambda_n S\lambda_n$ implies
$$\lambda_\infty\lambda_n=a_\infty^{-1}a_\infty\lambda_n=
a_\infty^{-1}a_\infty=\lambda_\infty$$ and similarly $\lambda_n\lambda_\infty=\lambda_\infty$, which means that $\lambda_\infty\le\lambda_n$.

Thus $\lambda_\infty$ is a lower bound of the chain $L$ and by Zorn Lemma, the set $\pi(A)$ has a minimal element $e\in E$. Let $b\in A$ be any element with $\pi(b)=e$. We claim that $b$ is topologically periodic. By the countable compactness of $S$, the sequence $\{b^n:n\in\IN\}$ has an accumulation point $c$, which belongs to the closed subsemigroup $A\cap eSe$ of $S$. It follows that $e\pi(c)=ecc^{-1}=cc^{-1}=\pi(c)$ and $\pi(c)e=c^{-1}ce=c^{-1}c=\pi(c)$, which means that $\pi(c)\le e$. By the minimality of $e$  in the poset $\pi(A)$, we get $\pi(c)=e$. Therefore, the point $c$ is an accumulation point of the sequence $(b^n)_{n=1}^\infty$ in the metrizable topological group $H_e$. Then be can choose an increasing number sequence $(n_k)_{k\in\w}$ such that $\lim_{k\to\infty} (n_{k+1}-n_k)=\infty$ and the sequence $(b^{n_k})_{k\in\w}$ converges to $c$. Then the sequence $(b^{n_{k+1}-n_k})_{k\in\w}$ tends to the idempotent $e=cc^{-1}$ and lies in the subsemigroup $A$.

Now we see that the idempotent $e$ belongs to the closed set $A$ of accumulating points of the sequence $(a^n)_{n\in\w}$. Since $S$ is first countable at $e$, there is an increasing number sequence $(m_k)_{k\in\w}$ such that the sequence $(a^{m_k})_{k\in\w}$ tends to $e$.  By the continuity of the inversion operation at $e$, the sequence $(a^{-m_k})_{k\in\w}$ also tends to $e$ and hence $e=ee^{-1}=\lim_{k\to\infty}a^{m_k}a^{-m_k}=aa^{-1}$. Since $\lim_{k\to\infty}a^{m_k}=e=aa^{-1}$, the sequence $(a^{m_k+1})_{k\in\w}$ tends to $aa^{-1}a=a$, witnessing that the element $a$ is topologically periodic.
\end{proof}

Claims~\ref{cl1}, \ref{cl4} and Theorem~\ref{t1.3}(4) imply

\begin{claim} $S$ is a topological Clifford semigroup.
\end{claim}

Being countably compact, the space $S$ is $M$-space and then it is metrizable by Theorem~\ref{t2.4}.
\end{proof}

\end{document}